\documentclass[reqno]{amsart}
\usepackage{amsmath, mathtools}
\usepackage{amssymb}

\usepackage[T1]{fontenc}
\usepackage[utf8]{inputenc}
\usepackage[english]{babel}

\usepackage{fancyhdr}
\usepackage{lmodern}
\usepackage{microtype}
\usepackage{geometry}
\usepackage{graphicx}
\usepackage{booktabs}
\usepackage{setspace}
\usepackage{amsmath,amssymb}
\usepackage{amsfonts}
\usepackage{amsthm}
\usepackage{amscd}
\usepackage{mathtools}
\usepackage{braket}
\usepackage{tikz}
\usepackage{tikz-cd}
\usepackage{enumitem}
\usepackage{mathrsfs}
\usepackage{epigraph}
\usepackage{afterpage}
\usepackage{extpfeil}
\usepackage{MnSymbol}
\usepackage{eqparbox}
\usepackage{scalerel}

\usetikzlibrary{graphs}
\usepackage[mathscr]{euscript}
\usepackage{stmaryrd}
 \let\mathscr\relax
\usepackage[scr]{rsfso}
\usepackage[all,cmtip]{xy}
\input xy
\xyoption{all}
\newtheorem{Lemma}{Lemma}[section]
\newtheorem{remark}[Lemma]{Remark}

\newtheorem{theorem}[Lemma]{Theorem}
\newtheorem{lemma}[Lemma]{Lemma}
\newtheorem{prop}[Lemma]{Proposition}
\newtheorem{corollary}[Lemma]{Corollary}
\newtheorem{definition}[Lemma]{Definition}

\newcommand{\Cal}[1]{{\mathcal #1}}

\newcommand{\Ann}{\operatorname{Ann}}
\newcommand{\Hom}{\operatorname{Hom}}

\newcommand{\Spec}{\operatorname{Spec}}
\newcommand{\Mod}{\operatorname{Mod-\!}}

\newcommand{\cmat}{\left(\begin{array}}
\newcommand{\fmat}{\end{array}\right)}
\DeclareMathOperator{\Ob}{Ob}

\usepackage{stackrel}

\begin{document}

\title{Exactness of cochain complexes via additive functors}
  \author[F.~Campanini and A.~Facchini]{Federico Campanini and Alberto Facchini}
\address{Dipartimento di Matematica ``Tullio Levi-Civita'', Universit\`a di Padova, 35121 Padova, Italy}
 \email{federico.campanini@phd.unipd.it, facchini@math.unipd.it}
\thanks{The second author is supported by Ministero dell'Istruzione, dell'Universit\`a e della Ricerca (Progetto di ricerca di rilevante interesse nazionale ``Categories, Algebras: Ring-Theoretical and Homological Approaches (CARTHA)'') and Dipartimento di Matematica ``Tullio Levi-Civita'' of Universit\`a di Padova (Research program DOR1828909 ``Anelli e categorie di moduli''). }

\begin{abstract} We investigate the relation between the notion of $e$-exactness, recently introduced by Akray and Zebary, and some functors naturally related to it, such as the functor $P\colon\Mod R\to \Spec(\Mod R)$, where $\Spec(\Mod R)$ denotes the spectral category of $\Mod R$, and the localization functor with respect to the singular torsion theory.
\end{abstract}
 \maketitle
    
\section{Introduction}

This paper has been inspired by the article \cite{Irakeni}, in which the
authors considered  the following type of ``$e$-exact'' cochain complexes
of modules over a ring $R$.

     \begin{definition}\label{1.1} Let $R$ be a ring, not necessarily
commutative. A cochain complex
   $$\ldots\longrightarrow
M^{i-1}\stackrel{f^{i-1}}{\longrightarrow}M^i\stackrel{f^i}{\longrightarrow}M^{i+1}\stackrel{f^{i+1}}{\longrightarrow}\ldots$$
of right $R$-modules and right $R$-module morphisms is said to be {\em
$e$-exact} if $f^{i-1}(M^{i-1})$ is an essential submodule of $\ker(f^i)$
for every $i$. In particular, a {\em short $e$-exact sequence} is a
cochain complex of the type
$0\longrightarrow A_R \stackrel{f}{\longrightarrow} B_R
\stackrel{g}{\longrightarrow} C_R \longrightarrow 0$ with $f$ a
monomorphism, $f(A_R)$ an essential submodule of $\ker(g)$, and $g(B_R)$
an essential submodule of $C_R$. \end{definition}

In \cite{Irakeni}, $R$ denoted a commutative ring. It is immediately
clear that the hypothesis of $R$ commutative is not necessary in this
definition, and that one can consider right modules over any ring $R$,
possibly noncommutative. Here, we study the relation
between this interesting notion and three functors naturally related to
it. The first functor is the functor $P\colon\Mod R\to \Spec(\Mod R)$
into the spectral category of the category $\Mod R$, as defined by
Gabriel and Oberst in \cite{GO}. The second functor is the localization
with respect to the singular torsion theory (Goldie topology) studied in
particular by Goodearl (\cite[Chapter~2]{goodearl} and
\cite[Examples~VI.6.2 and IX.2.2]{S}). This allows us to extend all the
results in the first part of \cite{Irakeni} from modules over
commutative rings to modules over arbitrary, possibly non-commutative,
rings, essentially replacing the usual torsion-theory with the singular
torsion-theory, and replacing commutative integral domains with right
non-singular rings. In particular, this applies to our third functor,
the functor  $-\otimes_R Q$, where $R$ is a right Ore domain and $Q$ is
the classical right ring of fractions of $R$. It is interesting to
notice, as we prove in Theorem~\ref{no_functors}, that for a ring $R$, except for the trivial case of $R$ artinian semisimple, there do not exist additive functors
$F\colon \Mod R\to\Cal A$ into any abelian category $\Cal A$ with the
property that a cochain complex $A_R \stackrel{f}{\longrightarrow} B_R
\stackrel{g}{\longrightarrow} C_R$ is $e$-exact in the sense of Definition~\ref{1.1} if and only if the cochain
complex $F(A_R) \stackrel{F(f)}{\longrightarrow} F(B_R)
\stackrel{F(g)}{\longrightarrow} F(C_R)$ is exact.

In the last section, we partially extend our setting to arbitrary
Gabriel topologies.

In the following, all rings $R$ are associative ring with identity
$1\ne0$, and  all  modules  are  unitary right $R$-modules.

\section{$e$-exactness and the spectral category}\label{s:1}

For any Grothendieck category  $\Cal A$, it is possible to construct the {\em spectral category} $\Spec(\Cal
A)$, which is a Grothendieck category obtained from $\Cal A$ by formally inverting all
essential monomorphisms of~$\Cal A$~\cite{GO}. More precisely, for any fixed object $A$ in $\Cal A$, the set of all the essential subobjects of $A$ is downward directed, because the intersection of two essential subobjects of $A$ is an essential subobject of $A$. We will write $A'\le_e A$ for ``$A'$ is an essential subobject of $A$''. If we fix another object $B$ of~$\Cal A$ and apply the contravariant functor $\Hom _{\Cal A}(-,B)$ to the essential subobjects $A'$ of $A$ and to the embeddings $A'\to A''$, where $A'\le_e A$, $A''\le_e A$ and $A'\subseteq A''$, we get an upward directed family of abelian groups $\Hom _{\Cal A}(A',B)$ and abelian group morphisms $\Hom _{\Cal A}(A'',B)\to \Hom _{\Cal A}(A',B)$. Take the direct limit $\varinjlim \Hom _{\Cal A}(A',B)$, where 
$A'$ ranges in the set of all essential subobjects of $A$. The \index{spectral category}{\em spectral category} $\Spec(\Cal A)$ of $\Cal A$ has the same
objects as $\Cal A$ and, for objects $A$ and $B$ of $\Cal A$,
$$
\Hom_{\Spec(\Cal A)}(A,B):=\varinjlim \Hom _{\Cal A}(A',B),
$$
where the direct limit is taken over all essential subobjects $A'$ of $A$.
The category $\Spec(\Cal A)$ can also be constructed as follows. Let $\Cal E$ be the full subcategory of $\Cal A$ consisting of all the injective objects of $\Cal A$. Let $\Cal I$ be the ideal of the category $\Cal E$, where, for every $E_R,F_R\in\Ob(\Cal E)$, $\Cal I(E_R,F_R)$ consists of all morphisms $E_R\to F_R$ whose kernel is essential in $E_R$. Then $\Spec(\Cal A)$ is the quotient category $\Cal E/\Cal I$. 
A third equivalent presentation of $\Spec(\Cal A)$ is as the category of fractions $\Cal A[S^{-1}]$ (see \cite{WIASC, Morfismos, GZ}), where $S$ is the class of all essential monomorphisms.
There is a canonical functor $P\colon \Cal A \to \Spec(\Cal A)$, which is the identity on objects. It is a left exact functor that sends essential monomorphisms to isomorphisms.

The category $\Spec(\Mod R)$ is a Grothendieck category in which every short exact sequence splits, i.e., in which every object is projective and injective. From the discussion made above, it is clear that every right $R$-module $M_R$ becomes isomorphic to its injective envelope $E(M_R)$ in $\Spec(\Mod R)$. In particular, $P(M_R)=0$ if and only if $M_R=0$. Moreover, a right $R$-module morphism $f\colon M_R\to N_R$ becomes a zero morphism in $\Spec(\Mod R)$ whenever $\ker(f)$ is an essential submodule of $M_R$.

If we view the morphisms $f\colon A_R\to B_R$ in $\Spec(\Mod R)$ as morphisms $f'\colon A'_R\to B_R$ for some essential submodule $A'_R$ of $A_R$, the description of the kernel, cokernel and image of $f$ is the following. The kernel of $f'$ in $\Mod R$ has a complement $C'_R$ in $A'_R$, that is, there exists $C'_R\le A_R$ with $\ker(f')\oplus C'_R$ essential in $A'_R$ (equivalently, $C'_R$ is maximal in the set of all submodules $X'_R$ of $A'_R$ with $\ker(f')\cap X'_R=0$). Then the submodule $f'(C'_R)\cong C'_R$ of $B_R$ has a complement $D'_R$ in $B_R$, i.e., $f'(C'_R)\oplus D'_R$ is essential in $B_R$. Then $\ker(f)$ is the kernel of $f$ in $\Spec(\Mod R)$, $D'_R$ is the cokernel of $f$ in $\Spec(\Mod R)$ and $f'(C'_R)$ is its image.

In the description of $\Spec(\Mod R)$ as the quotient category $\Cal E/\Cal I$, the kernel and the cokernel of a morphism $f\colon E_R\to F_R$ are as follows. Assume that $f$ is represented by a morphism $f'\colon E_R\to F_R$ in $\Mod R$. The kernel of $f'$ has a complement in $E_R$, that is, there exists $C_R\le E_R$ with $\ker(f')\oplus C_R$ essential in $E_R$ (it suffices to take $C_R$ maximal in the set of all submodules $X_R$ of $E_R$ with $\ker(f')\cap X_R=0$). Then $E_R$ decomposes as $E_R=E(\ker(f'))\oplus E(C_R)$. Moreover, $f'(E(C_R))\cong E(C_R)$ is an injective submodule of $F_R$, so we can write $F_R=f'(E(C_R))\oplus D_R$, for a suitable submodule $D_R$ of $F_R$. Then $E(\ker(f'))$ is the kernel of $f$ in $\Spec(\Mod R)$, $D_R$ is the cokernel of $f$ in $\Spec(\Mod R)$ and $f'(E(C_R))$ is the image of $f$.

\bigskip

We want to investigate the relation between the notion of $e$-exactness (Definition~\ref{1.1}) and the functor $P\colon \Mod R\to \Spec(\Mod R)$.
 \begin{lemma} Let
  $M^{i-1}\stackrel{f^{i-1}}{\longrightarrow}M^i\stackrel{f^i}{\longrightarrow}M^{i+1}$ be a cochain complex 
 of right $R$-modules and right $R$-module morphisms. The following conditions are equivalent:
 
 {\rm (a)} The sequence \begin{equation}P(M^{i-1})\stackrel{P(f^{i-1})}{\longrightarrow}P(M^i)\stackrel{P(f^i)}{\longrightarrow}P(M^{i+1})\label{P}\end{equation} is exact in $\Spec(\Mod R)$.
 
  {\rm (b)} If $C_R$ is a complement of $\ker(f^{i-1})$ in $M^{i-1}$, then $f^{i-1}(C_R)$ is essential in $\ker(f^i)$.\end{lemma}
  
  \begin{proof} 
  The kernel of $f^{i-1}$ in $\Mod R$ and the kernel of $P(f^{i-1})$ in $\Spec(\Mod R)$ coincide because $P$ is left exact. Similarly for the kernels of $f^{i}$ and $P(f^{i})$.
  
    Let $C_R$ be a complement of the kernel $\ker(f^{i-1})$ of $f$ in $\Mod R$, so that $\ker(f^{i-1})\oplus C_R$ is essential in $M^{i-1}$. Then $f^{i-1}(C_R)$ is the image of $P(f^{i-1})$. Hence the sequence (\ref{P}) is exact in $P(M^i)$ if and only if $f^{i-1}(C_R)$ is essential in $\ker(f^i)$.\end{proof}

From the previous lemma we immediately get the following proposition.

  \begin{prop} Let $0\longrightarrow A_R \stackrel{f}{\longrightarrow} B_R \stackrel{g}{\longrightarrow} C_R \longrightarrow 0$ be a cochain complex 
 of right $R$-modules and right $R$-module morphisms. The following conditions are equivalent:
 
 {\rm (a)} The sequence $0\longrightarrow P(A_R) \stackrel{P(f)}{\longrightarrow} P(B_R) \stackrel{P(g)}{\longrightarrow} P(C_R) \longrightarrow 0$ is a short exact sequence in $\Spec(\Mod R)$.
 
  {\rm (b)} $f$ is a monomorphism in $\Mod R$, $f(A_R)$ an essential submodule of $\ker(g)$, and if $B'_R$ is a complement of $f(A_R)$ in $B_R$, then $g(B'_R)$ is essential in $C_R$.\end{prop}
  
\begin{corollary}\label{2.3}
Let
\begin{equation}
0\longrightarrow A_R \stackrel{f}{\longrightarrow} B_R \stackrel{g}{\longrightarrow} C_R \longrightarrow 0\label{1}
\end{equation}
be a cochain complex of right $R$-modules and right $R$-module morphisms such that the sequence $0\longrightarrow P(A_R) \stackrel{P(f)}{\longrightarrow} P(B_R) \stackrel{P(g)}{\longrightarrow} P(C_R) \longrightarrow 0$ is a short exact sequence in $\Spec(\Mod R)$. Then the cochain complex $($\refeq{1}$)$ is a short $e$-exact sequence.\end{corollary}

\section{Homological lemmas with e-exact sequences}\label{section_singular}

The aim of this section is generalizing some results of \cite{Irakeni} to the non-commutative case. We will replace commutative domains with right non-singular rings and torsion [resp. torsionfree] modules with singular [resp. non-singular] modules. Let $R$ be a ring and let $\mathfrak{G}$ denote the family of all essential right ideals of $R$. Given a right $R$ module $M_R$, the singular submodule of $M_R$ is defined by $Z(M_R):=\{x \in M \mid xI=0 \mbox{ for some } I \in \mathfrak{G}\}=\{x\in M \mid \Ann_r(x) \in \mathfrak{G}\}$ \cite[Ch.~1, Sec.~D]{goodearl}. A module $M_R$ is called {\em singular} if $Z(M_R)=M_R$ and it is called {\em non-singular} if $Z(M_R)=0$. If $M$ is a module over a commutative domain $R$, then $Z(M)$ is equal to the torsion submodule $t(M)$ of $M$, so that $M$ is singular [resp. non-singular] if and only if $M$ is torsion [resp. torsionfree]. Notice that $Z(-)$ defines a functor $\Mod R \to \Mod R$. Indeed, given a right $R$-module morphism $f\colon M_R\to N_R$, we have $f(Z(M))\leq Z(N)$, and so a morphism $Z(f)\colon Z(M)\to Z(N)$. The functor $Z(-)$ is an idempotent preradical (cf. \cite[Chapter~VI]{S}), but it is not a radical in general. The smallest radical containing $Z$ is denoted by $Z_2$ and it is defined, for every right $R$-module $M_R$, by $Z_2(M)/Z(M)=Z(M/Z(M))$ \cite[Proposition~6.2]{S}. A ring $R$ is {\em right non-singular} if $Z_r(R):=Z(R_R)=0$ (notice that the dual situation of ``singular ring'' never occurs, since $1\notin Z(R_R)$). If $R$ is a right non-singular ring, then $Z(-)$ is actually a radical, that is, $Z(M/Z(M))=0$ for every right $R$-module $M_R$ (i.e., $Z_2(M)=Z(M)$ for every module $M_R$) \cite[Proposition~1.23]{goodearl}.

The following results extend  to the non-commutative case results proved in \cite{Irakeni}. A morphism $f\colon A\to B$ is said to be {\em $e$-epi} if $f(A)$ is essential in $B$.

\begin{prop} \label{4lemma}{\rm (4-lemma for $e$-exact sequences)}
Let $R$ be a right non-singular ring and consider the following commutative diagram of right  $R$-module morphisms with $e$-exact rows.
$$
\xymatrix{
A_1 \ar[r]^{f_1}\ar[d]^{t_1} &   A_2 \ar[r]^{f_2} \ar[d]^{t_2} & A_3 \ar[r]^{f_3} \ar[d]^{t_3} & A_4 \ar[d]^{t_4} \\
B_1 \ar[r]^{g_1} &   B_2 \ar[r]^{g_2} & B_3\ar[r]^{g_3} & B_4.
}
$$
\begin{enumerate}
\item[{\rm (1)}]
If $t_1,t_3$ are $e$-epis, $t_4$ is monic and $B_2$ is non-singular, then $t_2$ is $e$-epi.
\item[{\rm (2)}]
If $t_1$ is $e$-epi and $t_2,t_4$ are monic, then $\ker(t_3)$ is a singular module. In particular, if $A_3$ is a non-singular module, then $t_3$ is monic.
\end{enumerate}
\end{prop}

\begin{prop}\label{5lemma}{\rm  (5-lemma for $e$-exact sequences)} 
Let $R$ be a right non-singular ring and consider the following commutative diagram of right  $R$-module morphisms with $e$-exact rows
$$
\xymatrix{
A_1 \ar[r]^{f_1}\ar[d]^{t_1} &   A_2 \ar[r]^{f_2} \ar[d]^{t_2} & A_3 \ar[r]^{f_3} \ar[d]^{t_3} & A_4 \ar[r]^{f_4} \ar[d]^{t_4} & A_5 \ar[d]^{t_5}\\
B_1 \ar[r]^{g_1} &   B_2 \ar[r]^{g_2} & B_3\ar[r]^{g_3} & B_4 \ar[r]^{g_4} & B_5.
}
$$
where $B_3$ is a non-singular right $R$-modules.
\begin{enumerate}
\item[{\rm (1)}]
If $t_2$ and $t_4$ are $e$-epis, and $t_5$ is monic, then $t_3$ is $e$-epi.
\item[{\rm (2)}]
If $t_1$ is $e$-epi and $t_2,t_4$ are monic, then $\ker(t_3)$ is singular. In particular, if $A_3$ is non-singular, then $t_3$ is monic.
\end{enumerate}
\end{prop}

\begin{prop} {\rm ($3\times 3$-lemma for $e$-exact sequences)}\label{33lemma}
Let $R$ be a non-singular ring and consider the following commutative diagram of right $R$-modules
$$
\xymatrix{
 & 0 \ar[d] & 0 \ar[d] & 0 \ar[d] & \\
0\ar[r] & A_1 \ar[r]^{f_1} \ar[d]^{u_1} & A_2 \ar[r]^{f_2} \ar[d]^{v_1} & A_3 \ar[r] \ar[d]^{p_1} & 0 \\
0\ar[r] & B_1 \ar[r]^{g_1} \ar[d]^{u_2} & B_2 \ar[r]^{g_2} \ar[d]^{v_2} & B_3 \ar[r] \ar[d]^{p_2} & 0 \\
0\ar[r] & C_1 \ar[r]^{h_1} \ar[d] & C_2 \ar[r]^{h_2} \ar[d] & C_3 \ar[r] \ar[d] & 0 \\
 & 0 & 0 & 0 & 
}
$$
where $A_2$ and $A_3$ are non-singular modules. If the columns and the two bottom rows are $e$-exact, then the top row is also $e$-exact.
\end{prop}
 
In order to prove Propositions~\ref{4lemma}, \ref{5lemma} and \ref{33lemma}, recall a procedure to obtain a functor that, in some sense, can be viewed as a generalization of the ``localization functor'' for commutative rings to the non-commutative setting. We refer the reader to \cite[Chapter~2]{goodearl} for more details about this construction. Anyway, we prefer to adopt the notation used in \cite{S}, in which these topics are treated from a more general point of view (we will discuss this more general framework in Section~\ref{Section_Gabriel}).  Let $R$ be a right non-singular ring, that is, $Z_r(R)=0$. For any right $R$-module $A$, fix an injective envelope for $A/Z(A)$ and denote it by $A_{\mathfrak{G}}$. According to \cite[Proposition~1.23 (a)]{goodearl}, $A_\mathfrak{G}$ is a non-singular module. Moreover, \cite[Lemma~2.1]{goodearl} ensures that, for any given right $R$-module morphism $f\colon A_R \rightarrow B_R$, the induced morphism $\bar{f}:A/Z(A)\rightarrow B/Z(B)$ extends uniquely to a morphism $f_\mathfrak{G}\colon A_\mathfrak{G} \rightarrow B_\mathfrak{G}$. We have that $R_\mathfrak{G}$ is a ring containing $R$, and $(-)_\mathfrak{G} \colon \Mod R \rightarrow \Mod R_\mathfrak{G}$ turns out to be an exact functor.

\begin{Lemma}\label{S_functor}
Let $R$ be a right non-singular ring and let $f\colon A \rightarrow B$ be an $e$-epi right $R$-module morphism. Then
\begin{enumerate}
\item
$\bar{f}\colon A/Z(A)\rightarrow B/Z(B)$ is $e$-epi;
\item
$f_\mathfrak{G} \colon A_\mathfrak{G} \rightarrow B_\mathfrak{G}$ is surjective.
\end{enumerate}
\end{Lemma}

\begin{proof}
$(1)$ Let $b \in B \setminus Z(B)$. By hypothesis, there exists $I\leq_e R_R$ such that $0\neq bI \subseteq f(A)$, and therefore $(b+Z(B))I\subseteq \bar{f}(A/Z(A))$. Moreover, $bI \nsubseteq Z(B)$. Indeed, if $bI \subseteq Z(B)$, then $b+Z(B) \in Z(B/Z(B))=0$, which contradicts the fact that $b \notin Z(B)$. It follows that $\bar{f}$ is $e$-epi.

$(2)$ By $(1)$, $\bar{f}(A/Z(A))$ is essential in $B_\mathfrak{G}=E(B/Z(B))$, therefore $f_\mathfrak{G}(A_\mathfrak{G})\leq_e  B_\mathfrak{G} $. Since $B_\mathfrak{G}$ is a non-singular module, then so is $f_\mathfrak{G}(A_\mathfrak{G})$. In particular, $f_\mathfrak{G}(A_\mathfrak{G})$ is injective by \cite[Proposition~7.4]{S}, hence $f_\mathfrak{G} (A_\mathfrak{G})=B_\mathfrak{G}$.
\end{proof}

\begin{corollary}\label{2.4}
Let $R$ be a right non-singular ring. If 
$
0\longrightarrow A_R\longrightarrow B_R \longrightarrow C_R \longrightarrow 0
$
is a short $e$-exact sequence, then the sequence
$
0\longrightarrow A_\mathfrak{G} \longrightarrow B_\mathfrak{G} \longrightarrow C_\mathfrak{G}\longrightarrow 0
$
is exact.
\end{corollary}

\begin{proof} Since $(-)_{\mathfrak{G}}$ is an exact functor, it preserves kernels and images. By Lemma~\ref{S_functor}(2), if $f\colon M \rightarrow N$ is an $e$-epi right $R$-module morphism, then $f_\mathfrak{G} \colon M_\mathfrak{G} \rightarrow N_\mathfrak{G}$ is surjective.
\end{proof}

\begin{Lemma}\label{indietro}
Let $R$ be a right non-singular ring and let $f\colon A \rightarrow B$ be a right $R$-module morphism, where $B$ is non-singular. If $f_\mathfrak{G} \colon A_\mathfrak{G} \rightarrow B_\mathfrak{G}$ is surjective, then $f$ is $e$-epi.
\end{Lemma}

\begin{proof}
Assume, by contradiction, $f_\mathfrak{G} \colon A_\mathfrak{G} \rightarrow B_\mathfrak{G}$ surjective, but $f$ not $e$-epi. Then there exists a nonzero submodule $C$ of $B$ such that $f(A)\cap C=0$. Now $f$ factorizes as $f=\varepsilon\circ f'$, where $f'=f|^{f(A)}\colon A\to f(A)$ is the corestriction of $f$ and $\varepsilon\colon f(A)\hookrightarrow B$ is the inclusion. It follows that the image $f(A)$ of $\varepsilon$ is not essential in $B$, so that $E(B)=E(f(A))\oplus C'$ for some non-zero module $C'$. Applying the functor $(-)_\mathfrak{G}$ to the factorization $f=\varepsilon\circ f'$, we get that $f_\mathfrak{G}=\varepsilon_\mathfrak{G} \circ f'_\mathfrak{G}$. Now $f_\mathfrak{G}$ epi implies $\varepsilon_\mathfrak{G}$ epi. But since $f(A)\le B$ are non-singular modules, $\varepsilon_\mathfrak{G}$ is the splitting embedding of $E(f(A))$ into $E(B)$, which is not epi because $C'\ne0$.
\\
By assumption, $B\leq_e B_\mathfrak{G}=E(B)$. Fix $b \in B \setminus \{0\}$. Then there exists $\alpha \in A_\mathfrak{G}$ such that $f_\mathfrak{G}(\alpha)=b$. Moreover, there exists $I\leq_e R_R$ such that $\alpha I\subseteq A/Z(A)$. Hence for every $i \in I$ there exists $a^{(i)} \in A$ such that $\alpha i = a^{(i)}$. Thus $\bar{f}(a^{(i)}+Z(A))=bi+Z(B)=bi$, and so $f(a^{(i)})=bi$. It follows that $bI \subseteq f(A)$ and $bI \neq 0$, because $B$ is non-singular. Therefore $f$ is $e$-epi.
\end{proof}

We are now in a position to prove Propositions~\ref{4lemma}, \ref{5lemma} and \ref{33lemma}. Since all these results can be proved using the same procedure, we only write the proof of Proposition~\ref{4lemma}(1) and we leave the other proofs to the reader. 

Applying the functor $(-)_{\mathfrak{G}}$ to the diagram
$$
\xymatrix{
A_1 \ar[r]^{f_1}\ar[d]^{t_1} &   A_2 \ar[r]^{f_2} \ar[d]^{t_2} & A_3 \ar[r]^{f_3} \ar[d]^{t_3} & A_4 \ar[d]^{t_4} \\
B_1 \ar[r]^{g_1} &   B_2 \ar[r]^{g_2} & B_3\ar[r]^{g_3} & B_4
}
$$
we get a commutative diagram of $R_{\mathfrak{G}}$-modules with exact rows. Since $t_1$ and $t_3$ are $e$-epi, $(t_1)_{\mathfrak{G}}$ and $(t_3)_{\mathfrak{G}}$ are surjective by Lemma~\ref{S_functor}. The classical 4-lemma ensures that $(t_2)_{\mathfrak{G}}$ is surjective, hence $t_2$ is $e$-epi by Lemma~\ref{indietro}. This proves Proposition~\ref{4lemma}(1). The other proofs are similar.

\bigskip

We now want to discuss the previous situation in the particular case in which $R$ is a right Ore domain. A domain $R$ is said to be {\em right Ore} if $aR\cap bR\neq 0$ for every pair of nonzero elements $a,b \in R$, that is, if every nonzero (principal) right ideal of $R$ is essential. Given a right Ore domain $R$, it is possible to construct its classical right ring of fractions $Q:=Q_{cl}^r(R)$ (see \cite[Chapter~4]{Lam}). We have that $Q$ is a division ring which turns out to be a flat left $R$-module. Every element of $Q$ is of the form $rs^{-1}$ for suitable elements $r,s \in R$ and $s$ nonzero. Moverover, $R$ can be embedded into $Q$ via the ring homomorphism defined by $r \mapsto r1^{-1}$. For any right $R$-module $M_R$, we can consider the ``right localization'' $M(R\setminus\{0\})^{-1}\cong M\otimes_R Q$, which is a right $Q$-module with elements of the form $ms^{-1}$, for $m \in M$ and $s \in R\setminus\{0\}$. There is a natural map $M\rightarrow M(R\setminus\{0\})^{-1}$ whose kernel is the torsion submodule of $M$, $t(M):=\{m \in M \mid mr=0 \mbox{ for some nonzero }r \in R\}= Z(M_R)$. Since in a right Ore domain all nonzero right ideals are essential, $Q\cong R_{\mathfrak{G}}$ as rings and the functor $-\otimes_R Q$ is naturally isomorphic to $(-)_{\mathfrak{G}}$. Hence we have the following result.

\begin{Lemma}\label{lemma-Ore}
Let $R$ be a right Ore domain with classical right ring of fractions $Q$ and let $f\colon A_R\rightarrow B_R$ be a right $R$-module morphism.
\begin{enumerate}
\item
If $f$ is $e$-epi, then $A\otimes_R Q\stackrel{f\otimes 1}{\longrightarrow} B\otimes_R Q$ is surjective.
\item
If $A\otimes_R Q\stackrel{f\otimes 1}{\longrightarrow} B\otimes_R Q$ is surjective and $B$ is non-singular, then $f$ is $e$-epi.
\end{enumerate}
 In particular, if
$
0\to A_R\to B_R \to C_R \to 0
$
is a short $e$-exact sequence, then the sequence $0\to A_R\otimes_R Q \to B_R\otimes_R Q \to C_R \otimes_R Q\to 0
$ is exact.
\end{Lemma}

Corollaries \ref{2.3} and \ref{2.4} show that the notion of short $e$-exact sequence is intermediate between that of becoming exact applying the functor $P$ and that of becoming exact applying the functor $(-)_\mathfrak{G}$. The next result shows that there are no additive functors such that the notion of short $e$-exact sequence coincides with that of becoming exact applying $F$.

\begin{theorem}\label{no_functors}
Let $R$ be a ring that is not artinian semisimple. Then there do not exist additive functors $F\colon \Mod R\to\Cal A$ into any abelian category $\Cal A$ with the property that a cochain complex $A_R \stackrel{f}{\longrightarrow} B_R \stackrel{g}{\longrightarrow} C_R$ is $e$-exact if and only if the cochain complex $F(A_R) \stackrel{F(f)}{\longrightarrow} F(B_R) \stackrel{F(g)}{\longrightarrow} F(C_R)$ is exact.
\end{theorem}

\begin{proof}
Suppose that such a category $\Cal A$ and functor $F\colon \Mod R\to\Cal A$ exist. First of all, notice that the functor $F$ must be exact, because every exact sequence of $R$-modules is $e$-exact. Also, notice that if $f\colon A_R\to B_R$ is any essential monomorphism, then the cochain complex $0\longrightarrow A_R \stackrel{f}{\longrightarrow} B_R \longrightarrow 0$ is $e$-exact, so that the cochain complex $0\longrightarrow F(A_R) \stackrel{F(f)}{\longrightarrow} F(B_R) \longrightarrow 0$ is exact ($F$ is additive, hence maps zero objects to zero objects and zero morphisms to zero morphisms). It follows that $F$ maps essential monomorphisms to isomorphisms. Now for any singular module $C_R$, there exists an exact sequence $0\longrightarrow A_R \stackrel{f}{\longrightarrow} B_R \stackrel{g}{\longrightarrow} C_R \longrightarrow 0$ with $f$ an essential monomorphism \cite[Proposition 1.20(b)]{goodearl}. Applying the exact functor $F$, we get an exact sequence $0\longrightarrow F(A_R) \stackrel{F(f)}{\longrightarrow} F(B_R) \stackrel{F(g)}{\longrightarrow} F(C_R) \longrightarrow 0$ with $F(f)$ an isomorphism, so $F(C_R)=0$. This proves that $F$ necessarily annihilates all singular right $R$-modules. Now $R$ is not artinian semisimple, so it has a maximal right ideal that is not a direct summand \cite[Lemma 3.16]{LIBRO!}, hence there is a simple right $R$-module $S_R$ that is not projective. By \cite[Proposition~1.24]{goodearl}, $S_R$ is singular. Now consider the embedding $f\colon S_R\to S_R\oplus S_R$ into the first direct summand. The cochain complex $S_R\stackrel{f}{\longrightarrow} S_R\oplus S_R \longrightarrow 0$ is not $e$-exact, but, if we apply to it the functor $F$, we get the exact sequence $F(S_R)\stackrel{F(f)}{\longrightarrow} F(S_R\oplus S_R) \longrightarrow 0$. This contradicts the hypothesis on $F$ in the statement of the Theorem.
\end{proof}

\section{Gabriel topologies}\label{Section_Gabriel}

The results we have seen in Section~\ref{section_singular} for the functor $(-)_{\mathfrak{G}}$ can be extended to a much more general setting. Let $\mathfrak{F}$ be any right Gabriel topology on a ring $R$ \cite[Chapter VI \S 5]{S}. Thus $\mathfrak{F}$ corresponds to a hereditary torsion theory in $\Mod R$. Suppose this hereditary torsion theory {\em stable}, that is, the class of $\mathfrak{F}$-torsion modules is closed under injective envelopes \cite[Chapter VI \S 7]{S} and set for every right $R$-module $M_R$, $M_\mathfrak{F}:=\varinjlim \Hom_R(I, M_R)$, where the direct limit is taken over the downwards directed family of ideals $I \in\mathfrak{F}$. The assignment $M\mapsto M_\mathfrak{F}$ defines a left exact functor $\Mod R\to \Mod R_\mathfrak{F}$ (see \cite[Chapter~IX \S 1]{S}).
We say that a cochain complex 
  $$\ldots\longrightarrow M^{i-1}\stackrel{f^{i-1}}{\longrightarrow}M^i\stackrel{f^i}{\longrightarrow}M^{i+1}\stackrel{f^{i+1}}{\longrightarrow}\ldots$$ of right $R$-modules is  {\em 
  $\mathfrak{F}$-exact} if the cohomology modules $\ker(f^i)/ f^{i-1}(M^{i-1})$ are $\mathfrak{F}$-torsion $R$-modules for every $i$. That is, if $(f^{i-1}(M^{i-1}):x)\in\mathfrak{F}$ for every $i$ and every $x\in \ker(f^i)$.

\begin{Lemma}\label{1601} Let $\mathfrak{F}$ be a perfect stable Gabriel topology on $R$ and let $f\colon A_R\to B_R$ be a right $R$-module morphism.

{\rm (a)} If $B/f(A)$ is $\mathfrak{F}$-torsion, then the induced morphism $\widetilde{f}\colon A_{\mathfrak{F}}\to B_{\mathfrak{F}}$ is surjective.

{\rm (b)} Conversely, if the morphism $\widetilde{f}\colon A_{\mathfrak{F}}\to B_{\mathfrak{F}}$ is surjective and $B_R$ is $\mathfrak{F}$-torsionfree, then $B/f(A)$ is $\mathfrak{F}$-torsion.\end{Lemma}

\begin{proof} (a) Assume $B/f(A)$ is $\mathfrak{F}$-torsion. We have an exact sequence $A_R \stackrel{f}{\longrightarrow} B_R\to B/f(A)\to 0$. Since $\mathfrak{F}$ is perfect stable, the localization functor $(-)_{\mathfrak{F}}$ and the tensor product functor $-\otimes_RR_{\mathfrak{F}}$ are naturally isomorphic \cite[Proposition~XI.3.4(e)]{S} and $_RR_{\mathfrak{F}}$ is a flat left $R$-module \cite[Proposition~XI.3.11(b)]{S}. It follows that the induced sequence $A_{\mathfrak{F}} \to B_{\mathfrak{F}}\to (B/f(A)_{\mathfrak{F}}\to 0$ is exact. But $(B/f(A)_{\mathfrak{F}})= 0$ because $B/f(A)$ is $\mathfrak{F}$-torsion.

(b) Suppose $\widetilde{f}\colon A_{\mathfrak{F}}\to B_{\mathfrak{F}}$ surjective and $B_R$ $\mathfrak{F}$-torsionfree. There is a commutative square 
$$
\xymatrix{
A\ar[r]^{f}\ar[d]_{\psi_A} & B  \ar[d]^{\psi_B} & \\
A_{\mathfrak{F}} \ar[r]_{\widetilde{f}} &   B_{\mathfrak{F}} \ar[r]& 0.
}
$$ In order to prove that $B/f(A)$ is $\mathfrak{F}$-torsion, we must show that for every element $b\in B$ there exists $I\in\mathfrak{F}$ such that $bI\subseteq f(A)$. Now $\psi_B(b)\in B_{\mathfrak{F}}$ and $\widetilde{f}$ is onto, so that there exists an element $\widetilde{a}\in A_{\mathfrak{F}}$ with $\widetilde{f}(\widetilde{a})=\psi_B(b)$. Since $A_{\mathfrak{F}}=\displaystyle \lim_{\longrightarrow}\Hom(I,A_R)$ by \cite[Proposition~IX.1.7]{S}, the element $\widetilde{a}\in A_{\mathfrak{F}}$ is represented by a morphism $g\colon I'\to A_R$ in $\Hom(I',A_R)$ for some right ideal $I'\in\mathfrak{F}$. But $\widetilde{f}(\widetilde{a})=\psi_B(b)$ in $B_{\mathfrak{F}}$, so there exists a right ideal $I\subseteq I'$, $I\in\mathfrak{F}$, such that the morphisms $\rho_b\colon R\to B$ (right multiplication by $b$) and $fg\colon I'\to B$ have the same restriction to $I'$. Then, for every $i\in I$, we get that $bi=fg(i)$. Therefore $bI\subseteq f(A_R)$, as desired.
\end{proof}

\begin{Lemma}
Let $\mathfrak{F}$ be a perfect stable Gabriel topology on a ring $R$. If the cochain complex $A_R \stackrel{f}{\longrightarrow} B_R \stackrel{g}{\longrightarrow} C_R 
$ is $\mathfrak{F}$-exact, then the sequence $A_{\mathfrak{F}} \stackrel{f}{\longrightarrow} B_{\mathfrak{F}}  \stackrel{g}{\longrightarrow} C_{\mathfrak{F}}  
$ of right $R_{\mathfrak{F}} $-modules is exact.
\end{Lemma}

\begin{proof}
We have that $f=\varepsilon\circ f'$, where $\varepsilon\colon\ker g\to B$ is the inclusion and $f'\colon A_R\to\ker g$ is the corestriction of $f$ to $\ker g$ (which contains $f(A_R)$). Then $f'\colon A_R\to\ker g$ has an $\mathfrak{F}$-torsion cokernel, so that the induced mapping $\widetilde{f'}\colon A_{\mathfrak{F}}\to (\ker g)_{\mathfrak{F}}$ is onto by Lemma~\ref{1601}(a). Applying the exact functor $-\otimes R_{\mathfrak{F}}$ to the exact sequence $0\to \ker g \stackrel{\varepsilon}{\longrightarrow} B_R \stackrel{g}{\longrightarrow} C_R $, we get the exact sequence $0\to (\ker g)_{\mathfrak{F}} \stackrel{\widetilde{\varepsilon}}{\longrightarrow} B_{\mathfrak{F}} \stackrel{\widetilde{g}}{\longrightarrow} C_{\mathfrak{F}} $. Therefore $\ker(\widetilde{g})=\widetilde{\varepsilon}((\ker g)_{\mathfrak{F}} )=\widetilde{\varepsilon}(\widetilde{f'}(A_\mathfrak{F}))=\widetilde{f}(A_\mathfrak{F})$, as desired.
\end{proof}

As one might expect, it is also possible to prove suitable versions of the 4-lemma, 5-lemma and $3\times 3$-lemma in this setting. Here, we present the 4-lemma with $\mathfrak{F}$-exact sequences. The other statements (and their proofs) are the obvious ones.

\begin{prop}{\rm (4-lemma with $\mathfrak{F}$-exact sequences)}
Let $\mathfrak{F}$ be a perfect stable Gabriel topology on a ring~$R$, and consider the following commutative diagram of right  $R$-module morphisms with $\mathfrak{F}$-exact rows.
$$
\xymatrix{
A_1 \ar[r]^{f_1}\ar[d]^{t_1} &   A_2 \ar[r]^{f_2} \ar[d]^{t_2} & A_3 \ar[r]^{f_3} \ar[d]^{t_3} & A_4 \ar[d]^{t_4} \\
B_1 \ar[r]^{g_1} &   B_2 \ar[r]^{g_2} & B_3\ar[r]^{g_3} & B_4.
}
$$
\begin{enumerate}
\item[{\rm (1)}]
If $t_1,t_3$ are $\mathfrak{F}$-epis, $t_4$ is monic and $B_2$ is $\mathfrak{F}$-torsionfree then $t_2$ is $\mathfrak{F}$-epi.
\item[{\rm (2)}]
If $t_1$ is $\mathfrak{F}$--epi and $t_2,t_4$ are monic, then $\ker(t_3)$ is an $\mathfrak{F}$-torsion $R$-module. In particular, if $A_3$ is a $\mathfrak{F}$-torsionfree $R$-module, then $t_3$ is monic.
\end{enumerate}
\end{prop}

\begin{proof}
Apply the functor $-\otimes R_{\mathfrak{F}}$ to the commutative diagram, getting a commutative diagram of $R_{\mathfrak{F}}$-modules with exact rows. The classical 4-lemma and Lemma~\ref{1601} allow us to conclude.
\end{proof}

\begin{remark}{\rm
Let $R$ be a non-singular ring and let $\mathfrak{G}$ be the right Goldie topology, that is, the Gabriel topology given by the essential right ideals of $R$. In this case, the $\mathfrak{G}$-torsion $R$-modules are precisely the singular $R$-modules and the notion of $\mathfrak{G}$-exactness can be compared with that of $e$-exactness. Recall that, given a right $R$-module extension $A\leq B$, if $A\leq_e B$, then $B/A$ is a singular module, but the converse may fail to be true (even for abelian groups; see \cite[pp. 31-32]{goodearl}). This means that the notion of $e$-exactness is stronger than that of $\mathfrak{G}$-exactness, that is, given a cochain complex 
$$
\ldots\longrightarrow M^{i-1}\stackrel{f^{i-1}}{\longrightarrow}M^i\stackrel{f^i}{\longrightarrow}M^{i+1}\stackrel{f^{i+1}}{\longrightarrow}\ldots
$$
of right $R$-modules, if the cochain if $e$-exact, then it is also $\mathfrak{G}$-exact, but the converse is not true in general.
}
\end{remark}

\end{document}